\documentclass[reqno, 12pt]{amsart}

\numberwithin{equation}{section}

\usepackage{amsfonts,amssymb,amsmath,amsthm}
\usepackage{enumerate}
\usepackage{bbm}
\usepackage{times}
\usepackage{amsfonts}
\usepackage{amssymb}
\usepackage{bbm}
\usepackage{times}
\usepackage{amsfonts}
\usepackage{amssymb}
\usepackage{bbm}
\usepackage{times}

\usepackage{enumerate}
\usepackage{amsmath}
\usepackage{amsthm}
\usepackage{color}

\newtheorem{theor}{Theorem}[section]
\newtheorem{cor}[theor]{Corollary}
\newtheorem{lemma}[theor]{Lemma}

\newtheorem{prop}[theor]{Proposition}

\newcounter{other}            
\newtheorem{otherth}[other]{Theorem}              

\def \Hu{\mathcal{H}_\mu}

\def \D{\mathbb{D}}

\def \N{\mathbb{N}}
\def \T{\mathbb{T}}
\def \d{\mathcal D}

\def \H{\mathcal{H}}

\def \C {\mathcal C}

\def  \bm {\boldsymbol}
\def \X {\mathcal{X}}
\def \Ce{\mathcal{C}_{\bm{\eta}}}

\begin{document}
\title[Hankel matrices]
{Hankel matrices acting on the Dirichlet space}

\author{ Guanlong Bao, Kunyu Guo, Fangmei Sun and Zipeng Wang }

\address{Guanlong Bao\\
    Department of Mathematics, Shantou University\\
    Shantou, Guangdong, 515821, China}
\email{glbao@stu.edu.cn}

\address{Kunyu Guo\\
    School of Mathematical Sciences, Fudan University\\
    Shanghai, 200433, China}
\email{kyguo@fudan.edu.cn}

\address{Fangmei Sun\\
    School of Mathematical Sciences, Fudan University\\
    Shanghai, 200433, China}
\email{fmsun@foxmail.com}

\address{Zipeng Wang\\
    College of Mathematics and Statistics, Chongqing University\\
    Chongqing, 401331, China}
\email{zipengwang2012@gmail.com}

\subjclass[2010]{47B35,  30H20, 46E15}
\keywords{Hankel matrix, Ces\`aro matrix, Carleson measure, Dirichlet space, Bergman space}

\begin{abstract}
The characterization of the boundedness of  operators induced by  Hankel matrices on analytic function spaces can be traced back to the work of Z. Nehari  and H. Widom on the Hardy space, and has been extensively studied on many other analytic function spaces  recently. However, this  question  remains open in the context of the  Dirichlet space \cite{GGMM}. By  Carleson measures, the Widom type condition and the reproducing kernel thesis, this paper provides a comprehensive solution to this   question. As a beneficial product, characterizations  of  the  boundedness and compactness  of operators induced by  Ces\`aro  type matrices on the Dirichlet space are   given. In addition, we also show  that a random Dirichlet function  almost surely induces  a compact  Hankel type operator on the Dirichlet space.
\end{abstract}

\maketitle

\section{Introduction}

Let $\mathbb{N}$ be the set of nonnegative integers. Suppose $\boldsymbol{\lambda}=\{\lambda_n\}_{n\in\mathbb{N}}$ is a sequence of complex numbers. An infinite Hankel matrix $\bm{H}_{\boldsymbol{\lambda}}$ induced by the sequence $\boldsymbol{\lambda}$ is given by
$
\bm{H}_{\boldsymbol{\lambda}}=(\lambda_{j+k})_{j,k\in\mathbb{N}}.
$
The Hankel matrix $\bm{H}_{\boldsymbol{\lambda}}$ is initially defined for all finitely supported sequences in $\ell^2$. The celebrated Nehari theorem \cite{Ne} illustrates that $\bm{H}_{\boldsymbol{\lambda}}$ represents a bounded operator on $\ell^2$ if and only if there exists a function $\psi$ in $L^\infty$ on the unit circle $\mathbb{T}$ such that $\lambda_n$, $n\geq 0$, is the $n$-th Fourier coefficient of $\psi$.

For $f(z)=\sum_{k=0}^\infty a_kz^k$ in $H(\D)$, the class of functions analytic in the open unit disk $\D$, the Hankel matrix $\bm{H}_{\boldsymbol{\lambda}}$ acts on the function  $f$ via
\begin{equation*}\label{Hop}
\H_{\boldsymbol{\lambda}}(f)(z):=\sum_{n=0}^\infty\left(\sum_{k=0}^\infty \lambda_{n+k}a_k\right)z^n, \ \ z\in \D,
\end{equation*}
whenever the right-hand side makes sense and defines an analytic function in $\D$.  For a space $X\subseteq H(\D)$  equipped with a norm $\|\cdot\|_X$, the Hankel type operator $\H_{\boldsymbol{\lambda}}$
is bounded on $X$ if $\H_{\boldsymbol{\lambda}}(f)$ is well defined in $H(\D)$ for any $f\in X$ and there exits a positive constant $C$ independent of $g$ such that $\|\H_{\boldsymbol{\lambda}}(g)\|_X \leq C \|g\|_X$ for all $g$ in $X$.

 By the Fourier transform,  the  Hankel matrix $\bm{H}_{\boldsymbol{\lambda}}$ represents a bounded operator on  $\ell^2$  if and only if $\H_{\boldsymbol{\lambda}}$ is bounded on the Hardy space $H^2$. See \cite{Pel} for details.

When ${\boldsymbol{\lambda}}$ is the moment sequence $\{\mu_n\}_{n\in\mathbb{N}}$ of a finite positive Borel measure $\mu$ on  $[0, 1)$, where $\mu_n=\int_{[0, 1)}t^n d\mu(t)$,   the related Hankel matrix is denoted by $\bm{H}_\mu$ in the literature. Write $\H_\mu$ for the corresponding Hankel type operator.
H. Widom \cite{Wi} proved that $\H_\mu$ is bounded on $H^2$ if and only if $\mu_j=O(\frac{1}{j+1})$.
It is worth noting that $(\frac{1}{j+1})_{j\in\mathbb{N}}$ is the moment sequence of the Lebesgure measure on $[0,1)$, which corresponds to the classical Hilbert matrix $\bm{H}=((j+k+1)^{-1})_{j,k\in\mathbb{N}}$. See \cite{BK, D2, DA, DJV, Ka, Ka1, LMN, LMW, LMW1, LNP, YF} for developments of the Hilbert matrix acting on analytic function spaces.

Since then, operator questions regarding the Hankel matrix have been investigated in various other function spaces of analytic functions on the unit disk, including Hardy spaces, weighted Bergman spaces, Dirichlet type spaces, and Möbius invariant function spaces. In \cite{D1},  E. Diamantopoulos' work  illustrates   that H. Widom's condition remains true for the Dirichlet-type space $\mathcal{D}_\alpha$ with $0<\alpha<2$. He also noted that the boundedness of  $\H_\mu$ on the  classical Dirichlet space $\d$ and the Bergman space $A^2$ coincides. Under a milder condition, using Carleson measures for $\d$, P. Galanopoulos and J. Pel\'aez \cite{GP} completely characterized  the boundedness of $\H_\mu$ on the Dirichlet space.   C. Chatzifountas, D. Girela and J. Pel\'aez  studied the boundedness and compactness of $\H_\mu$ between distinct Hardy spaces in \cite{CGP}. D. Girela and  N. Merch\'an \cite{GM} developed a method to give complete descriptions  of the boundedness and compactness of $\H_\mu$  on some Hardy spaces and M\"obius invariant function spaces. G. Bao, F. Ye and K. Zhu \cite{BYZ} considered bounded  $\H_\mu$ on analytic functions spaces in terms of so-called Hankel measures.

In their recent work \cite{GGMM}, P. Galanopoulos, D. Girela, A. Mas, and N. Merch\'an demonstrated  that a 1-logarithmic 1-Carleson measure ensures the boundedness of $\H_\mu$ on the Dirichlet space. However, they also pointed out that  the characterization of the  finite positive Borel measure $\mu$ on  $[0, 1)$  for which the operator  $\H_\mu$  is bounded on the Dirichlet space remains unresolved.
Via  Carleson measures, Widom type conditions and the reproducing kernel thesis, the current work  answers this question.

Recall that  the  Dirichlet space  $\mathcal D$  is a Hilbert space of analytic functions on $\D$ equipped  with the Dirichlet inner product
\begin{equation}\label{inner product D}
\langle f, g\rangle_{\mathcal D}=f(0)\overline{g(0)}+\int_\D f'(z) \overline{g'(z)} dA(z).
\end{equation}
Here $dA(z)=\pi^{-1}dxdy$ for $z=x+iy$ is the normalized Lebesgue  measure on $\D$.  A  finite positive  Borel measure  $\nu$ on $\D$ is said to be a Carleson measure for the Dirichlet space if   there is a positive constant  $C$ such that
\begin{equation}\label{CM for diri}
\left(\int_\D |f(z)|^2 d\nu(z)\right)^{\frac 12}\leq C \|f\|_\d
\end{equation}
for all $f\in \d$. Denote by $CM(\d)$ the set of Carleson measures for $\d$.  The smallest such  $C$ in (\ref{CM for diri}) is denoted by $\|\nu\|_{CM(\d)}$, the Carleson measure norm of $\nu$. If the identity map  $I_d: \d\to L^2(\D, d\nu)$ is compact, then we say that $\nu$ is a vanishing Carleson measure for $\d$. See \cite{ARS, arswbook, EKMR, GP, S, Wu0} for this  definition of  Carleson measures for $\d$ and related investigations.

In this paper, we first obtain the following result.
\begin{theor}\label{main00}
 Suppose $\boldsymbol{\lambda}=\{\lambda_n\}_{n\in\mathbb{N}}$ is a sequence of complex numbers. Then the Hankel type operator  $\H_{\bm\lambda}$ is bounded on the Dirichlet space $\mathcal{D}$ if and only if $h_{\overline{\bm \lambda}}$ is analytic on  $\mathbb{D}$ and the measure $|h_{\overline{{\bm \lambda}}}'(z)|^2dA(z)$ is a Carleson measure for the Dirichlet space $\d$, where $h_{\overline{{\bm \lambda}}}(z)=\sum_{n=0}^\infty \overline{\lambda_n}z^n$.
\end{theor}

In the case of $\H_\mu$,
by Theorem \ref{main00}, the operator  $\H_\mu$ is bounded on $\d$ if and only if $|h'_\mu(z)|^2dA(z)$ is a Carleson measure for $\d$, where
$$h_\mu(z)=\int_{[0, 1)}\frac{1}{1-tz}d\mu(t).$$
Note that in  \cite{S} D. Stegenga's beautiful characterization  of $CM(\d)$  is related to  the logarithmic capacity
of a finite union of intervals of the unit circle.  Based on certain integrals involve the Carleson box and the
heightened box,   N. Arcozzi, R. Rochberg and E. Sawyer  \cite[Theorem 1]{ARS} gave  a complete description of $CM(\d)$.  There still lacks a direct relation between moment sequences and the boundedness of $\H_\mu$ on the Dirichlet space.
We will establish these relations through the use of Widom type conditions and the reproducing kernel thesis.

Corresponding  to the Dirichlet inner product (\ref{inner product D}),   the reproducing kernel of the  Dirichlet space $\mathcal{D}$ at a point $w\in\D$ is
$$
K_w(z)=1+ \log\frac{1}{1-z\overline{w}}
$$
and its normalized reproducing kernel
$$
k_w(w)=\frac{K_w(z)}{\sqrt{K_w(w)}}.
$$
Now, we state our characterization of the bounded Hankel type operator $\H_\mu$ on the Dirichlet spaces as follows.

\begin{theor}\label{Hu openanswer1}
Suppose  $\mu$ is a finite positive Borel measure on $[0,1)$. Then the following conditions are equivalent.
\begin{enumerate}
   \item [(i)]  The Hankel type operator  $\H_\mu$ is bounded on $\mathcal{D}$.
   \item [(ii)] The reproducing kernel thesis holds; that is,
   $$
  \sup_{t\in[0,1)}\|\Hu (k_t)\|_{\mathcal{D}}<\infty,
   $$
   where $k_t$ is the normalized reproducing kernel of $\d$ at $t$ in $[0,1)$.
   \item [(iii)] The Widom type condition is true; that is,
   $$
   \sum^\infty_{n=m}n \mu_n^2=O\left(\frac{1}{\log (m+2)}\right).$$
\end{enumerate}
\end{theor}

To prove Theorem  \ref{Hu openanswer1},  we need Ces\`aro type matrices. As benefit products, the  necessary and sufficient condition for the boundedness of   Ces\`aro type  matrix acting on the  Dirichlet space is obtained.

Let ${\bm \eta}=\{\eta_n\}_{n\in\mathbb{N}}$ be a  sequence of complex numbers. Recall that the Ces\`aro type  matrix $\bm{C}_{\bm{\eta}}$ is the following infinite lower triangular matrix:
$$
\bm{C}_{\bm{\eta}}=
\begin{pmatrix}
\eta_0 & 0 & 0& 0 & \cdots\\
\eta_1 & \eta_1 & 0 & 0 & \cdots\\
\eta_2 & \eta_2 & \eta_2 & 0 & \cdots\\
\vdots & \vdots & \vdots & \ddots & \ddots
\end{pmatrix}.
$$
In the same way as operators induced by Hankel matrices, the matrix
$\bm{C}_{\bm{\eta}}$ induces a Ces\`aro type  operator $\Ce$   as follows:
 \begin{align}\label{b23}
\Ce (f)(z)=\sum_{n=0}^\infty\Big(\eta_n \sum_{k=0}^n a_k\Big)z^n, \quad z\in\D,
\end{align}
for $f(z)=\sum_{n=0}^\infty a_nz^n$ in  $H(\D)$, whenever the right-hand side of (\ref{b23}) makes sense and defines an analytic function in $\D$.

If $\eta_n=1/(n+1)$ for every $n$, then the Ces\`aro type  matrix  $\bm{C}_{\bm{\eta}}$  is the classical  Ces\`aro matrix   $\bm{C}$. In \cite{DS},  N. Danikas and A.  Siskakis showed  that the operator induced by  $\bm{C}$ is bounded from $H^\infty$ to $BMOA$. A.  Siskakis \cite{S1, S2} studied the action of $\bm{C}$ on Hardy spaces.

Given a finite positive Borel measure  $\mu$ on $[0,1)$, if
$$\eta_n=\mu_n=\int_{[0, 1)}\frac{1}{1-tz}d\mu(t)$$
for each  $n$, then we write $\Ce$ as $\C_\mu$. The operator $\C_\mu$ is said to be the  Ces\`aro-like operator which was introduced   in
\cite{GGM, JT} recently. We refer to \cite{BSW, Bla, GGMM, GGM1} for more results about    $\C_\mu$  on some spaces of analytic functions. From Theorem 5 in \cite{GGMM}, if $\mu$ is a 1-logarithmic 1-Carleson measure, then $\C_\mu$  is a bounded on $\d$. Conversely, if  $\C_\mu$  is  bounded on $\d$, then $\mu$ is a $\frac{1}{2}$-logarithmic 1-Carleson measure.
We obtain  the following result.
\begin{theor}\label{Ce 1}
Suppose ${\bm \eta}=\{\eta_n\}_{n=0}^\infty$ is a sequence of complex numbers. Then the following conditions are equivalent.
\begin{enumerate}
   \item [(i)] The Ces\`aro type operator   $\C_{\bm{\eta}}$ is bounded on $\mathcal{D}$.
   \item [(ii)]
  The reproducing kernel thesis holds; that is,
   $$
   \sup_{t\in[0,1)}\|\Ce k_t\|_{\mathcal{D}}<\infty,
   $$
  where $k_t$ is the normalized reproducing kernel of $\d$ at $t$ in $[0,1)$.
   \item [(iii)]
   The Widom type condition is true; that is,
   $$
   \sum^\infty_{n=m}n |\eta_n|^2=O\left(\frac{1}{\log (m+2)}\right).
   $$
\end{enumerate}
\end{theor}

Following our previous work on the boundedness of Hankel type operators and Ces\`aro type operators  on the Dirichlet space, we also  establish their compactness counterparts. These results are included in Section 4.

\vspace{0.1truecm}

Let $(X_n)_{n\in\mathbb{N}}$ be a sequence of independent identically distributed (i.i.d.) real random variables on a probability space $(\Omega,\mathcal{A},\mathbb{P})$. For a sequence $\{\alpha_n\}$ of complex numbers, one formally defines a random power series
\[
\varphi_\omega(z)=\sum_{n=0}^\infty X_n(\omega)\alpha_n z^n.
\]
By Z. Nehari's theorem, almost surely the random analytic function $\varphi_\omega$ induces a bounded Hankel type operator  on the Hardy space $H^2$ if and only if the symbol function $\varphi_\omega$ is a BMOA function almost surely. The famous unsolved J. Anderson's  question \cite{an} is to seek the necessary and sufficient condition to characterize random BMOA functions. One can consult \cite{Kah} for recent progress and more details in this subject.

According to Theorem \ref{main00}, the symbol function $h_{\overline{{\bm \lambda}}}$ of a bounded Hankel type operator  $\H_{\bm\lambda}$ on $\mathcal{D}$ must be in the Dirichlet space. However, our next corollary illustrates that, under milder conditions, a random Dirichlet function almost surely induced a compact  Hankel type operator, and hence almost surely induced a bounded Hankel type operator.

\begin{theor}\label{random}
	Suppose $\boldsymbol{\lambda}=\{\lambda_n\}_{n\in\mathbb{N}}$ is a sequence of complex numbers.
		Let $(X_n)_{n\in\mathbb{N}}$ be a sequence of i.i.d. real random variables satisfying that  $\mathbb{E}[X_n]=0$ and $\mathbb{E}[X_n^4]$ is finite.
	If $h_{\bm \lambda}\in\mathcal{D}$, then almost surely the Hankel type operator $\H_{\bm{\omega}}$  is compact on $\mathcal{D}$,
where $\boldsymbol{\omega}=\left\{X_n \overline{\lambda_n}\right\}_{n\in\mathbb{N}}$.
\end{theor}

 Let $E\subseteq\mathbb{Z}$ be a set. As defined by W. Rudin in \cite{rudin}, the integrable function $f$ on the unite circle $\T$ is called an $E$-function if $\hat{f}(n)$, the $n$-th Fourier coefficient of $f$, is equal to zero  for all $n\in\mathbb{Z}\setminus E$. Suppose $p>2$, the set $E\subseteq\mathbb{Z}$ is called a Rudin's  $\Lambda(p)$ set  if there is a constant $C$ such that
 $
 \|f\|_{L^p(\T)}\leq C\|f\|_{L^2(\T)}
 $
 for all trigonometric polynomials whose coefficients are equal to $0$ on $\mathbb{Z}\setminus E$. The Rudin's $\Lambda(p)$ set is a natural generalization of the classical lacunary sequences (the Hadamard set) which plays important roles in many aspects of analysis (see \cite{rudin}, \cite{GK} and \cite{Bou} for more details). Recall that an increasing sequence $\{n_k\}_{k\in \N}$ of positive integers   is said to be  lacunary if there exists $q$ in $(1, \infty)$ such that $n_{k+1}/n_k>q$ for all $k$. It is also known that there exist  Rudin's  $\Lambda(p)$ sets rather than Hadamard sets.

Our next result presents a characterization of bounded and compact Hankel type operators  on $\mathcal{D}$ when the non zero Fourier coefficients of the symbol function are located in a  Rudin's $\Lambda(p)$ set.

\begin{cor}\label{rudin}
	Suppose $p>2$, the set $E$ is a  Rudin's $\Lambda(p)$ set and
	$\boldsymbol{\lambda}=\{\lambda_{n}: n\in E\}$  is a sequence of complex numbers. Then the following statements are equivalent.
	\begin{enumerate}
		\item [(i)] The Hankel type operator  $\H_{\boldsymbol{\lambda}}$ is compact on $\d$.
		\item [(ii)] The Hankel type  operator  $\H_{\boldsymbol{\lambda}}$ is  bounded  on $\d$.
		\item [(iii)] The symbol function $h_{\overline{{\bm \lambda}}}(z)=\sum_{n\in E} \overline{\lambda_n}z^n$ belongs to $\d$; that is,    $$
		\sum_{n\in E}  (n+1) |\lambda_{n}|^2<\infty.
		$$
	\end{enumerate}
\end{cor}

Throughout  this paper, the symbol $A\thickapprox B$ means that $A\lesssim
B\lesssim A$. We say that $A\lesssim B$ if there exists a positive
constant $C$ such that $A\leq CB$.

\section{Bounded Hankel type operators $\H_{\boldsymbol{\lambda}}$  on $\mathcal D$}

In this  section we  will prove Theorem \ref{main00}. We also consider the M\"obius invariance of the norm of bounded $\H_{\boldsymbol{\lambda}}$  on the Dirichlet space.   For $f\in H(\D)$, we say that $f\in \mathcal{X}$ if
$$
\|f\|^2_{\mathcal{X}}=|f(0)|^2+\||f'|^2dA\|_{CM(\d)}<\infty.
$$
Denote by $\X_0$ the norm closure in $\X$ of the space of polynomials. Equivalently, $f\in \X_0$ if and only if $|f'|^2dA$ is a vanishing Carleson measure for $\d$. We refer to  \cite{ARSW0, Wu} for the theory of
$\X$ and $\X_0$.

Given a function $b$ in $H(\D)$,  define  a Hankel type bilinear form on the Dirichlet space, initially for $f$, $g$ in $\mathcal{P}$, as
$$
T_b(f, g):=\langle fg, b\rangle_{\mathcal D}.
$$
Here $\mathcal{P}$ is the space of polynomials.
The norm of the bilinear form $T_b$ is
$$
\|T_b\|_{\d\times \d}=\sup\left\{|T_b(f, g)|: \ \ \|f\|_\d=\|g\|_\d=1, \ f, g\in  \mathcal{P}\right\}.
$$
For a bounded bilinear form $T_b$ on $\d$,  $T_b$ is said to   be compact on $\d$ if $T_b(B_\d\times B_\d)$ is precompact in $\mathbb{C}$, where $B_\d=\{f\in \d: \|f\|_\d\leq 1\}$. In other words,  $T_b$ is compact on $\d$ if and only if for all bounded sequences $\{(f_n, g_n)\}\subseteq \d\times \d$, the sequence
$\{T_b(f_n, g_n)\}$ has  a convergent subsequence. See  \cite{BT} for more descriptions on compact  bilinear operators.

With respect to the   normalized basis  $\{(n+1)^{-1/2}z^n\}$ of the Dirichlet space $\d$, the matrix representations of  $T_b$   is
$$
\left(\frac{j+k}{\sqrt{j+1}\sqrt{k+1}}\overline{b_{j+k}}\right), \ \ \ \ j, k=0, 1, \cdots,
$$
where $b_{j+k}$ is the $(j+k)$-th  Taylor coefficient of $b$.
N. Arcozzi, R. Rochberg, E. Sawyer and B. Wick  characterized the boundedness and compactness of  the bilinear form on the Dirichlet space as follows.

\begin{otherth}\cite[Theorem 1.1]{ARSW}\label{Bi D}
Let $b\in H(\D)$. Then the following  assertions hold.
\begin{enumerate}
   \item [(i)] $T_b$ extends to a bounded bilinear form  on $\d$ if and only if $b\in \X$.
   \item [(ii)] $T_b$ extends to a compact  bilinear form  on $\d$ if and only if $b\in \X_0$.
   \end{enumerate}
\end{otherth}

\begin{proof}[Proof of Theorem \ref{main00}]
Given a sequence of complex number ${\bm \lambda}=\{\lambda_n\}^\infty_{n=0}$, suppose  $h_{\overline{{\bm \lambda}}}$ is analytic  on  $\mathbb{D}$ and the measure $|h_{\overline{{\bm \lambda}}}'(z)|^2dA(z)$ is a  Carleson measure for $\d$. Then
$h_{\overline{\bm \lambda}}\in \X$. It follows that $h_{\bm \lambda}\in \d$, and
$$\|h_{\bm \lambda}\|^2_\d=|\lambda_0|^2+\sum_{n=0}^\infty (n+1)|\lambda_{n+1}|^2<\infty.$$
Let $f(z)=\sum^\infty_{k=0} c_k z^k$ belong to  $\d$.
Then for any $j\in\mathbb{N}$,
\begin{align*}
\left|\sum_{k=0}^\infty \lambda_{j+k}c_k\right| &\lesssim \|f\|_\d \bigg(\sum_{k=0}^\infty \frac{|\lambda_{j+k}|^2}{k+1}\bigg)^{1/2}\\
&\lesssim  \|f\|_\d \bigg(\sum_{k=0}^\infty (j+k+1)|\lambda_{j+k}|^2\bigg)^{1/2}\\
&\lesssim  \|f\|_\d \|h_{\bm \lambda}\|_\d.
\end{align*}
Write  $a_j=\sum_{k=0}^\infty \lambda_{j+k}c_k$.
Then $\{a_j\}_{j\in \mathbb{N}}$ is a bounded sequence of  complex numbers   and hence  $\H_{\bm \lambda}(f)(z)=\sum_{j=0}^\infty a_jz^j$ is analytic on $\mathbb{D}$.
Note that
\[
\|\H_{\bm \lambda}f\|_\d^2= |a_0|^2+\sum_{j=0}^{\infty}(j+1)|a_{j+1}|^2.
\]
Let $m$ be a positive integer. For any polynomial $g(z)=\sum^m_{j=0}b_j z^j$ with $\|g\|_\d\leq 1$, we have
\begin{align*}
& \bigg|a_0 \overline{b_0}+ \sum_{j=0}^{m-1} (j+1)a_{j+1}\overline{b_{j+1}}\bigg|\\
\leq & |\lambda_0| |c_0| |b_0|+|b_0|\Big|\sum^\infty_{k=1}\lambda_{k}c_{k}\Big|+\Big|\sum^{m-1}_{j=0} \overline{b_{j+1}} (j+1)\sum^\infty_{k=0}\lambda_{j+1+k}c_k\Big|\\
= & |\lambda_0| |c_0| |b_0|+ |b_0| \left|\int_\D\frac{f(z)-f(0)}{z}\overline{h'_{\overline{\bm \lambda}} (z)} dA(z)\right|\\
&+\left|\int_\D f(z)g'_1(z)\ \overline{h'_{\overline{\bm \lambda}} (z)} dA(z)\right|\\
\leq& |\lambda_0| \|f\|_{\mathcal{D}} \|g\|_{\mathcal{D}}+\|g\|_{\mathcal{D}}\Big(\int_\D\Big|\frac{f(z)-f(0)}{z}\Big|^2|h'_{\overline{\bm \lambda}} (z)|^2 dA(z)\Big)^{\frac12}\\
& + \Big(\int_\D |f(z)|^2|h'_{\overline{\bm \lambda}} (z)|^2 dA(z)\Big)^{\frac12}\Big(\int_\D \left|g_1'(z)\right|^2 dA(z)\Big)^{\frac12},
\end{align*}
where $g_1(z)=\sum^m_{j=0}\overline{b_j} z^j$.
Since $|h'_{\overline{\bm \lambda}} (z)|^2 dA(z)$ is a  Carleson measure for $\d$, there exists a positive  constant $C$ such that
\[
\bigg|a_0 \overline{b_0}+ \sum_{j=0}^{m-1} (j+1)a_{j+1}\overline{b_{j+1}}\bigg|\leq C\|f\|_\d.
\]
Therefore
\[
\bigg(|a_0|^2+\sum_{j=0}^{m-1}(j+1)|a_{j+1}|^2\bigg)^{1/2}\leq C\|f\|_\d
\]
and $\|\H_{\bm \lambda}f\|_\d\leq C\|f\|_\d.$

On the other hand, suppose the operator  $\H_{\bm \lambda}$   is bounded on $\mathcal{D}$. Then for any $f$ and $g$ in $\d$,
$$
|\langle\H_{\bm \lambda} f,g\rangle_{\mathcal{D}}|\lesssim \|f\|_\d \|g\|_\d.
$$
This implies that  $h_{\overline{{\bm \lambda}}}$ is analytic in $\mathbb{D}$ and
$$
\int_{\mathbb{D}}|h_{\overline{{\bm \lambda}}}'(z)|^2dA(z)<\infty.
$$
Moreover, for any polynomial $f$ and $g$,
\begin{align}\label{b1}
\int_\D f(z)g'(z)\ \overline{h'_{\overline{\bm \lambda}} (z)} dA(z)=\int_\D (\H_{\bm \lambda} f)'(z) g'(\overline{z}) dA(z),
\end{align}
and
\begin{align}\label{b2}
\int_\D f'(z) g(z)\overline{h'_{\overline{\bm \lambda}} (z)} dA(z)=\int_\D (\H_{\bm \lambda} g)'(z) f'(\overline{z})dA(z).
\end{align}
Consequently,
$$
\left|f(0)g(0)\overline{h_{\overline{\bm \lambda}} (0)}\right|+\left|\int_\D (f(z)g(z))'\overline{h'_{\overline{\bm \lambda}} (z)} dA(z)\right|\lesssim \|f\|_\d \|g\|_\d.
$$
Therefore, $T_{h_{\overline{\bm \lambda}}}$ is a bounded bilinear form   on $\d$.
It follows from  Theorem \ref{Bi D}  that  $h_{\overline{\bm \lambda}}\in \X$ and  $|h_{\overline{{\bm \lambda}}}'(z)|^2dA(z)$ is a Carleson measure for $\d$. This completes the whole proof of Theorem \ref{main00}.
\end{proof}

Denote by $\text{Aut}(\D)$ the  M\"obius group which consists  of all one-to-one  analytic functions that map  $\D$ onto itself. Let $\widetilde{\d}$ be the space of functions $f$ in $\d$ with $f(0)=0$.
For the bounded $\H_{\bm\lambda}$  on  $\widetilde{\d}$,    the norm of $\H_{\bm\lambda}$  is M\"obius invariant in the following sense.

\begin{prop}
 Let  $\boldsymbol{\lambda}=\{\lambda_n\}_{n\in\mathbb{N}}$ be  a sequence of complex numbers. Suppose  the Hankel type operator  $\H_{\bm\lambda}$ is bounded on $\widetilde{\d}$. Then there exist positive constants $C_1$ and $C_2$
 depending only on  $\boldsymbol{\lambda}$ such that
 $$
 C_1 \|\H_{\bm\lambda}\| \leq \|\H_{\bm{\lambda_\phi}}\|\leq C_2 \|\H_{\bm\lambda}\|
 $$
 for all $\phi \in \text{Aut}(\D)$, where $\bm{\lambda_\phi}=\left\{\overline{\lambda_{\phi, n}}\right\}_{n\in\mathbb{N}}$  with
 $$
 h_{\overline{\bm \lambda}}\circ \phi(z)= \sum_{n=0}^\infty \lambda_{\phi, n} z^n.
 $$
\end{prop}
\begin{proof}
By Theorem \ref{main00}, the boundedness of  $\H_{\bm\lambda}$  on  $\widetilde{\d}$ yields
$$
\||h'_{\overline{\bm \lambda}}|^2dA\|_{CM(\widetilde{\d})}<\infty.
$$
From the proof of Theorem \ref{main00}, there are  positive constants $C_1$ and $C_2$ depending only on  $\boldsymbol{\lambda}$ such that
\begin{equation}\label{Y30}
C_1\||h'_{\overline{\bm \lambda}}|^2dA\|_{CM(\widetilde{\d})}\leq  \|\H_{\bm\lambda}\|\leq C_2 \||h'_{\overline{\bm \lambda}}|^2dA\|_{CM(\widetilde{\d})}.
\end{equation}
 Let  $\phi \in \text{Aut}(\D)$. For $g\in \widetilde{\d}$, by the change of  variables, we see
\begin{align*}
 &\int_\D |g(z)|^2 |(h_{\overline{\bm \lambda}}\circ \phi)'(z)|^2dA(z)\\
 =& \int_\D   |g(\phi^{-1}(w))|^2   |h'_{\overline{\bm \lambda}}(w)|^2 dA(w).
 \end{align*}
 Also,
 $$
 \int_\D |g'(z)|^2 dA(z)= \int_\D   |(g\circ\phi^{-1})'(w)|^2  dA(w).
 $$
 Thus there is $C>0$ such that
 $$
 \int_\D |g(z)|^2 |(h_{\overline{\bm \lambda}}\circ \phi)'(z)|^2dA(z)\leq C \int_\D |g'(z)|^2 dA(z)
 $$
 for all $g\in \widetilde{\d}$  if and only if
 $$
 \int_\D   |k(w)|^2   |h'_{\overline{\bm \lambda}}(w)|^2 dA(w)\leq C \int_\D |k'(z)|^2 dA(z)
 $$
 for all $k\in \widetilde{\d}$. This means
 $$
 \||h'_{\overline{\bm \lambda}}|^2dA\|_{CM(\widetilde{\d})}=\| |(h_{\overline{\bm \lambda}}\circ \phi)'|^2 dA\|_{CM(\widetilde{\d})}.
 $$
 Combining this with (\ref{Y30}), we get the desired result.
\end{proof}

\section{Bounded Hankel type operators  $\H_\mu$ and Ces\`aro type  operators    on the Dirichlet space }
This section is devoted to an intrinsic description of the  boundedness  of $\H_\mu$ on $\d$. Indeed, we can characterize bounded Hankel type operators  $\H_\mu$    on $\d$ induced  by a decreasing sequence of positive numbers. The proof is based on our  characterization of bounded Ces\`aro type  operators   on $\d$.

The following Hilbert's double theorem for the Dirichlet space plays a crucial role in our proofs. By using Schur's test, one can show that inequality (see the proof of Theorem 2 in \cite{Sh} for more details).

\begin{otherth}\cite[p. 814]{Sh}\label{Hilbert D}
Let $f(z)=\sum_{k=0}^\infty a_k z^k$ in $\d$. Then there exists a positive constant $C$ independent of $f$ such that
\begin{equation*}
\sum^\infty_{m=1}\sum_{n=1}^\infty \frac{|a_n| |a_m|}{\log(n+m+1)}\leq C\sum^\infty_{n=1} n|a_n|^2.
\end{equation*}
\end{otherth}

We first prove Theorem \ref{Ce 1}.

\begin{proof}[Proof of Theorem \ref{Ce 1}]
Since $\|k_w\|_\d=1$ for any $w\in\mathbb{D}$, we obtain the implication $(i)\Rightarrow (ii)$.

Next, for each $t\in[0,1)$, recall that the normalized reproducing kernel $k_t$ for the Dirichlet space
$$
k_t(z)=\left(1+\log \frac{1}{1-t^2}\right)^{-\frac{1}{2}}\left(1+\sum_{n=1}^\infty \frac{1}{n} t^n z^n\right).
$$
By direct computations,
\begin{align*}
&\| \Ce k_t\|^2_{\mathcal{D}}\\
= & \left(1+\log \frac{1}{1-t^2}\right)^{-1}\left(|\eta_0|^2+\sum^\infty_{n=0}(n+1)|\eta_{n+1}|^2\Big(1+\sum^{n+1}_{k=1}\frac{t^{k}}{k}\Big)^2\right).
\end{align*}
Thus for any positive integer $m$ and $t\in[0,1)$,
\begin{align*}
 \| \Ce k_t\|^2_{\mathcal{D}}
\gtrsim&  \left(\log\frac{e}{1-t}\right)^{-1}\sum^\infty_{n=m}(n+1)|\eta_{n+1}|^2\Big(\sum^m_{k=1}\frac{t^{k}}{k}\Big)^2\\
\gtrsim& \left(\log\frac{e}{1-t}\right)^{-1} t^{2m}(\log(m+1))^2\sum^\infty_{n=m}(n+1)|\eta_{n+1}|^2.
\end{align*}
In particular, let $t=\frac{m}{m+1}\in(0,1)$,
$$
\frac{1}{\log (m+1)} \gtrsim \sum^\infty_{n=m} (n+1)|\eta_{n+1}|^2.
$$
By (ii), we have $\| \Ce k_t\|^2_{\mathcal{D}}\lesssim 1$ and
   $$
   \sum^\infty_{n=m}n |\eta_n|^2=O\left(\frac{1}{\log (m+2)}\right),
   $$
which gives $(ii)\Rightarrow (iii)$.

$(iii)\Rightarrow (i)$.  The Widom type condition gives
$$|\eta_n|^2\lesssim\frac{1}{(n+1)\log(n+2)},$$
for all  $n\in \mathbb{N}$.
For  $f(z)=\sum^\infty_{k=0}a_k z^k$ in $\mathcal{D}$,  we get
\begin{align*}
\left|\eta_n\sum^n_{k=0}a_k\right|\leq & |\eta_n|\left(\sum^n_{k=0}\frac{1}{k+1}\right)^{\frac12}\left(\sum^n_{k=0}(k+1)|a_k|^2\right)^{\frac12}\\
\leq &|\eta_n|\left(\log(n+2)\right)^{\frac12} \|f\|_{\mathcal{D}}\\
\lesssim & \|f\|_{\mathcal{D}}
\end{align*}
for all  nonnegative integers $n$.
Thus $\Ce(f)$ is well defined in $H(\D)$.

Because of  $(iii)$,
\begin{align}\label{b27}
\sum^\infty_{n=0}(n+1)|\eta_{n+1}|^2<\infty.
\end{align}
Let  $g(z)=\sum^\infty_{k=0}a_k z^k$ be in the Dirichlet space $\mathcal{D}$. Then
\begin{align*}
&\|\Ce g\|^2_{\mathcal{D}}\\
\leq  & |\eta_0 a_0|^2 +2 \sum^\infty_{n=0}(n+1)|\eta_{n+1}|^2 |a_0|^2+ 2\sum^\infty_{n=0}(n+1)|\eta_{n+1}|^2\Big( \sum^{n+1}_{k=1} |a_k| \Big)^2.
\end{align*}
Using (\ref{b27}), (iii),  and  Theorem \ref{Hilbert D}, we deduce
\begin{align*}
&\|\Ce g\|^2_{\mathcal{D}}\\
\lesssim & \|g\|_\d^2+ \small{\sum^\infty_{n=0}(n+1)|\eta_{n+1}|^2 \bigg(\sum^{\infty}_{k=1} |a_k| \chi_{\{k\leq n+1\}}(k) \bigg)\bigg(\sum^{\infty}_{j=1} |a_j|\chi_{\{j\leq n+1\}}(j) \bigg)} \\
\thickapprox &  \|g\|_\d^2 +\sum^\infty_{k=1} \sum^\infty_{j=1} |a_k| |a_j|\bigg(\sum^\infty_{n=\max\{k-1,j-1\}}(n+1)|\eta_{n+1}|^2\bigg)\\
\lesssim &  \|g\|_\d^2 + \sum^\infty_{k=1} \sum^\infty_{j=1} \frac{|a_k| |a_j|}{\log(k+j+1)}\\
\lesssim & \|g\|^2_{\mathcal{D}},
\end{align*}
where $\chi_{\{k\leq n+1\}}(k)$ is the characteristic function on  ${\{k\leq n+1\}}\cap \N^+$. Here $\N^+$ is the set of positive integers. Hence $\Ce$ is bounded on $\mathcal{D}$.  The proof of Theorem \ref{Ce 1} is complete.
\end{proof}

The proof of Theorem \ref{Hu openanswer1} follows from the next more general result.

\begin{theor}\label{H1}
Suppose ${\bm \lambda}=\{\lambda_n\}^\infty_{n=0}$  is a decreasing sequence of positive real numbers. Then the following conditions are equivalent.
\begin{enumerate}
   \item [(i)] The Hankel type operator  $\H_{\bm \lambda}$ is bounded on $\mathcal{D}$.
   \item [(ii)]
   The reproducing kernel thesis holds; that is,
   $$
  \sup_{t\in[0,1)}\|\H_{\bm \lambda} k_t\|_{\mathcal{D}}<\infty,
   $$
    where $k_t$ is the normalized reproducing kernel of $\d$ at $t$ in $[0,1)$.
   \item [(iii)]
   The Widom type condition is true; that is,
   $$
   \sum^\infty_{n=m}n \lambda_n^2=O\left(\frac{1}{\log (m+2)}\right).
   $$
\end{enumerate}
\end{theor}

\begin{proof}
If  $\H_{\bm \lambda}$ is bounded on $\mathcal{D}$, it is clear that (ii) holds.

$(ii)\Rightarrow (iii)$. Since   $\{\lambda_n\}$  is a decreasing sequence of positive numbers, 
\begin{align*}
\|\H_{\bm \lambda} k_t\|_{\mathcal{D}}^2
\geq & \left(1+\log \frac{1}{1-t^2}\right)^{-1}\sum^\infty_{n=1} (n+1)  \Big(\sum^\infty_{k=1}\lambda_{n+k+1}\frac{t^{k}}{k}\Big)^2 \\
\geq&\left(1+\log \frac{1}{1-t^2}\right)^{-1}\sum^\infty_{n=m} (n+1)  \Big(\sum^n_{k=1}\lambda_{n+k+1}\frac{t^{k}}{k}\Big)^2 \\
\geq& \left(1+\log \frac{1}{1-t^2}\right)^{-1}\sum^\infty_{n=m} (n+1) \lambda_{2n+1}^2 \Big(\sum^m_{k=1}\frac{t^{k}}{k}\Big)^2 \\
\gtrsim& \left(\log \frac{e}{1-t}\right)^{-1} t^{2m}(\log(m+1))^2 \sum^\infty_{n=m}(n+1)\lambda_{2n+1}^2,
\end{align*}
for all $t\in [0, 1)$ and all  positive integers $m$.

For any fixed integer $m$, taking  $t=\frac{m}{m+1}$ in the above, we get
$$
 \sum^\infty_{n=m}(2n+1)\lambda_{2n+1}^2\lesssim \frac{1}{\log (m+1)}
$$
and
$$
 \sum^\infty_{n=m}(2n+2)\lambda_{2n+2}^2\lesssim \frac{1}{\log (m+1)}.
$$
Then  the desired result holds.

$(iii)\Rightarrow (i)$. It follows from  (iii) that
\begin{align}\label{q1}
 \lambda_k^2 \lesssim \frac{1}{k \log (k+1)}
\end{align}
for all positive integers $k$.
Let  $f(z)=\sum^\infty_{k=0}a_k z^k$ belong to $\mathcal{D}$. By (\ref{q1}) and the monotonicity of $\{\lambda_k\}$,
\begin{align*}
\Big|\sum^\infty_{k=0}\lambda_{n+k}a_k\Big|
\leq & \lambda_n  |a_0| +\sum^\infty_{k=1}\lambda_{n+k}|a_k| \\
\lesssim  & \|f\|_{\mathcal{D}}+\|f\|_{\mathcal{D}}\Big(\sum^\infty_{k=1}\frac{\lambda^2_{n+k}}{k+1}\Big)^{\frac12}\\
\lesssim & \|f\|_{\mathcal{D}}+ \|f\|_{\mathcal{D}}\Big(\sum^\infty_{k=1}\frac{1}{(k+1)^2\log(k+1)}\Big)^{\frac12}\\
\lesssim & \|f\|_{\mathcal{D}}
\end{align*}
for all  nonnegative integers $n$.
Then  $\H_{\bm \lambda}(f)$ is analytic on the unit disk $\mathbb{D}$ for any  $f$ in $\d$.

Observe that by condition (iii), we have
\begin{align}\label{q2}
\sum^\infty_{n=0}(n+1)\lambda_{n+1}^2<\infty.
\end{align}
Then the H\"older inequality gives
$
\sum^\infty_{k=0} \lambda_k |a_k|\lesssim \|f\|_\d.
$
Consequently,
\begin{align}\label{q3}
\|\H_{\bm \lambda} f\|^2_{\mathcal{D}}\leq  &   \left(\sum^\infty_{k=0} \lambda_k |a_k|\right)^2  + \sum^\infty_{n=0}(n+1)\Big( \sum^\infty_{k=0}\lambda_{n+k+1} |a_k| \Big)^2 \nonumber \\
\lesssim & \|f\|_\d^2+ \sum^\infty_{n=0}(n+1)\Big(\sum^{n+1}_{k=0} \lambda_{n+k+1} |a_k|  \Big)^2 \nonumber \\
&+\sum^\infty_{n=0}(n+1)\Big(\sum^\infty_{k=n+1} \lambda_{n+k+1} |a_k|  \Big)^2.
\end{align}

By  Theorem \ref{Ce 1} and the monotonicity of the sequence $\bm \lambda$, it is true that
\begin{align}\label{q4}
&\sum^\infty_{n=0}(n+1)\Big(\sum^{n+1}_{k=0} \lambda_{n+k+1} |a_k|  \Big)^2 \nonumber \\
\leq & \sum^\infty_{n=0}(n+1) \lambda_{n+1}^2 \Big(\sum^{n+1}_{k=0}  |a_k|  \Big)^2 \nonumber \\
\leq & \|\C_{\bm \lambda} f_2\|^2\lesssim \|f\|^2_\d,
\end{align}
 where $f_2(z)=\sum^\infty_{k=0}|a_k| z^k$ with    the same Dirichlet norm of $f$.

 Note that  the monotonicity of the sequence $\bm \lambda$ again,  the H\"older inequality and (\ref{q2}). Then
 \begin{align}\label{q5}
 &\sum^\infty_{n=0}(n+1)\Big(\sum^\infty_{k=n+1} \lambda_{n+k+1} |a_k|  \Big)^2  \nonumber \\
 \leq & \sum^\infty_{n=0}(n+1)\Big(\sum^\infty_{k=n+1} \lambda_{k} |a_k|  \Big)^2  \nonumber \\
 \lesssim & \|f\|^2_\d \sum^\infty_{n=0}(n+1) \bigg( \sum^\infty_{k=n+1} \frac{\lambda^2_k}{k}\bigg)\nonumber \\
 \thickapprox & \|f\|^2_\d  \sum^\infty_{k=1}\frac{\lambda^2_k}{k} \sum^{k-1}_{n=0}(n+1) \nonumber \\
  \thickapprox & \|f\|^2_\d\sum^\infty_{k=1} k \lambda^2_k  \lesssim  \|f\|^2_\d.
 \end{align}
 Joining (\ref{q3}), (\ref{q4}), and (\ref{q5}), we get the boundedness of $\H_{\bm \lambda}$ on $\d$. The proof is complete.
\end{proof}

\section{Compact Hankel and Ces\`aro type  operators   on $\d$}

In this section, we give corresponding  results about the compactness of  Hankel and Ces\`aro type operators   on $\d$.

The result below is the compact version of Theorem \ref{main00}.
\begin{theor} \label{Main00c}
Suppose $\boldsymbol{\lambda}=\{\lambda_n\}_{n\in\mathbb{N}}$ is a sequence of complex numbers. Then the Hankel type operator  $\H_{\bm\lambda}$ is compact  on the Dirichlet space $\mathcal{D}$ if and only if $h_{\overline{{\bm \lambda}}}$ is analytic on  $\mathbb{D}$ and the measure $|h_{\overline{{\bm \lambda}}}'(z)|^2dA(z)$ is a vanishing Carleson measure for the Dirichlet space $\d$.
\end{theor}
\begin{proof}
 Suppose $h_{\overline{\bm \lambda}}\in \X_0$.  Then the identity map  $$I_d: \d\to L^2(\D, |h'_{\overline{\bm \lambda}}|^2dA)$$ is compact.  Let $\{f_k\}_{k=1}^\infty$ be a bounded sequence in $\d$ such that $\{f_k\}_{k=1}^\infty$ tends
to 0 uniformly in compact subsets of $\D$ as $k\to \infty$.
From the proof of Theorem \ref{main00}, we obtain
\begin{align*}
& |\langle\H_{\bm \lambda} f_k,g\rangle_{\mathcal{D}}|\\
\lesssim & |\lambda_0| |f_k(0)| |g(0)|+\|g\|_{\mathcal{D}}\Big(\int_\D\Big|\frac{f_k(z)-f_k(0)}{z}\Big|^2|h'_{\overline{\bm \lambda}} (z)|^2 dA(z)\Big)^{\frac12}\\
& + \|g\|_\d \Big(\int_\D |f_k(z)|^2|h'_{\overline{\bm \lambda}} (z)|^2 dA(z)\Big)^{\frac12}
\end{align*}
for all $g\in \d$. Hence, for any $\varepsilon>0$,  there is an integer $k_0$ such that
$$
|\langle\H_{\bm \lambda} f_k,g\rangle_{\mathcal{D}}|\lesssim \|g\|_\d \ \  \varepsilon
$$
for all $k>k_0$ and all $g\in \d$. Then $\|\H_{\bm \lambda} f_k\|_\d \to 0$ as $k\to \infty$. Thus   $\H_{\bm \lambda}$ is compact  on $\mathcal{D}$.

Conversely, suppose $\H_{\bm \lambda}$ is compact  on $\mathcal{D}$.  For a bounded sequence  $\{(f_n, g_n)\}_{n=1}^\infty \subseteq \d\times \d$, both $\{f_n\}$ and $\{g_n\}$ are bounded sequences in $\d$.
Then both  $\{\H_{\bm \lambda}(f_n)\}$ and $\{\H_{\bm \lambda}(g_n)\}$ have convergent subsequences.
Because of (\ref{b1}) and (\ref{b2}), $T_{h_{\overline{\bm \lambda}}}$ extends to a compact  bilinear form  on $\d$.
Then  Theorem \ref{Bi D} yields $h_{\overline{\bm \lambda}}\in \X_0$. The proof is complete.
\end{proof}

For the compactness of $\C_{\bm{\eta}}$ on $\d$, we also have the following conclusion.
\begin{theor}\label{Ce 2}
Suppose ${\bm \eta}=\{\eta_n\}_{n=0}^\infty$ is a sequence of complex numbers. Then the following conditions are equivalent.
\begin{enumerate}
   \item [(i)] The Ces\`aro type operator $\Ce$ is compact  on $\mathcal{D}$.
    \item [(ii)] The reproducing kernel thesis holds; that is, $$
  \lim_{t\to 1^{-}}\|\Ce k_t\|_{\mathcal{D}}=0,
   $$
    where $k_t$ is the normalized reproducing kernel of $\d$ at $t$ in $[0,1)$.
   \item [(iii)]  The Widom type condition is true; that is,   $$
   \sum^\infty_{n=m}n |\eta_n|^2=o\left(\frac{1}{\log (m+2)}\right).
   $$
\end{enumerate}
\end{theor}
\begin{proof}
$(i)\Rightarrow (ii)$.  Note that $\|k_t\|_\d=1$ for each $t\in [0, 1)$ and $k_t$ tends to zero uniformly in compact subsets of $\D$ as  $t\to 1^{-}$. Then (ii) holds.

$(ii)\Rightarrow (iii)$. Checking the proof of Theorem \ref{Ce 1}, we see
\begin{align}\label{b28}
\log(m+1)  \sum^\infty_{n=m}(n+1)|\eta_{n+1}|^2 \lesssim \| \Ce k_t\|^2_{\mathcal{D}}
\end{align}
for all positive integers $m$ and all $t$ in $[0, 1)$. Taking $t\to 1^{-}$ in (\ref{b28}), we obtain (iii).

$(iii)\Rightarrow (i)$.  Let $m$ be a positive integer. For $f(z)=\sum_{n=0}^\infty a_nz^n$ in  $\d$, consider
 $$
\Ce^{(m)} (f)(z)=\sum_{n=0}^m \Big(\eta_n \sum_{k=0}^n a_k\Big)z^n, \quad z\in\D.
$$
Then $\Ce^{(m)}$ is  a  finite rank operator. Thus,  $\Ce^{(m)}$ is compact on $\d$. Because of (iii), for every $\epsilon>0$, there exists a positive integer $N$ such that
$$
\sum^\infty_{n=m} (n+1)|\eta_{n+1}|^2<  \frac{\epsilon}{\log (m+2)}
$$
for $m>N$. Using  the proof of Theorem \ref{Ce 1}, we have
\begin{align*}
&\sum^\infty_{n=m} (n+1) |\eta_{n+1}|^2 \Big( \sum^{n+1}_{k=1} |a_k| \Big)^2\\
 \lesssim & \sum^\infty_{k=1} \sum^\infty_{j=1} |a_k| |a_j|\sum^\infty_{n=\max\{k-1,j-1, m\}}(n+1)|\eta_{n+1}|^2\\
 \lesssim & \epsilon \sum^\infty_{k=1} \sum^\infty_{j=1}  \frac{|a_k| |a_j|}{\log(k+j+m)}\\
  \lesssim & \epsilon \|f\|_\d^2
\end{align*}
for all $m>N$.  Thus, for $m>N$,
\begin{align*}
&\|(\Ce-\Ce^{(m)}) (f)\|^2_\d\\
\leq & |a_0|^2\sum^\infty_{n=m} (n+1)|\eta_{n+1}|^2+\sum^\infty_{n=m} (n+1) |\eta_{n+1}|^2 \Big( \sum^{n+1}_{k=1} |a_k| \Big)^2\\
\leq &  \epsilon\|f\|_\d^2.
\end{align*}
In other words,   $\|\Ce-\Ce^{(m)}\|\to 0$  as $m\to \infty$. Hence  $\Ce$ is compact  on $\mathcal{D}$. The proof is finished.
\end{proof}

The  compact result  corresponding to Theorem \ref{H1} also holds.

\begin{theor}\label{H2}
Suppose ${\bm \lambda}=\{\lambda_n\}^\infty_{n=0}$  is a decreasing sequence of positive real numbers. Then the following conditions are equivalent.
\begin{enumerate}
   \item [(i)] The Hankel type  operator $\H_{\bm \lambda}$ is compact  on $\mathcal{D}$.
   \item [(ii)]  The reproducing kernel thesis holds; that is,  $$
 \lim_{t\to 1^-}\|\H_{\bm \lambda} k_t\|_{\mathcal{D}}<\infty,
   $$
    where $k_t$ is the normalized reproducing kernel of $\d$ at $t$ in $[0,1)$.
   \item [(iii)]  The Widom type condition is true; that is,  $$
   \sum^\infty_{n=m}n \lambda_n^2=o\left(\frac{1}{\log (m+2)}\right).
   $$
\end{enumerate}
\end{theor}
\begin{proof}
By   the proof of  Theorem \ref{H1}, the arguments here are  similar to  that  of Theorem \ref{Ce 2}. We omit it.
\end{proof}

\section{Random Hankel type operators on $\d$ and a result related to   Rudin's $\Lambda(p)$ set}

In this section,  we  prove Theorem \ref{random} and Corollary \ref{rudin}. Many known real random variables sequences $(X_n)_{n\in\mathbb{N}}$ satisfy the conditions in Theorem \ref{random}. First of all, all bounded mean zero random variables are included in Thereom \ref{random}. The typical example is the
Bernoulli random variables sequence $(X_n)_{n\in\mathbb{N}}$. In other words,  $(X_n)_{n\in\mathbb{N}}$ is independent such that $\mathbb{P}(X_n=1)=\mathbb{P}(X_n=-1)=1/2$. In addition, for the classical independent normal distribution sequence $X_n=N(0,1)$, we have $\mathbb{E}[X_n]=0$ and $\mathbb{E}[X_n^4]=3$ (cf. \cite[p. 2]{Le}). Therefore, Theorem \ref{random}  can be applied to the corresponding random Gaussian analytic Dirichlet functions.

For the $\boldsymbol{\lambda}=\{\lambda_n\}_{n\in\mathbb{N}}$ complex number sequence and a sequence of i.i.d. real random variables
$(X_n)_{n\in\mathbb{N}}$ with $\mathbb{E}[X_n]=0$ and $\mathbb{E}[X_n^4]<\infty$.
If $h_{\bm \lambda}\in\mathcal{D}$, then $\sum_{n=0}^\infty |\lambda_n|^2<\infty$. Moreover, by i.i.d. and $\mathbb{E}[X_n^2]\leq \sqrt{\mathbb{E}[X_n^4]}$, we have
\[
\mathbb{E}\big[\sum_{n=0}^\infty |\lambda_n X_n|^2\big]=\sum_{n=0}^\infty|\lambda_n|^2\mathbb{E}[|X_n|^2]<\infty
\]
and almost surely
\[
h_{\bm \lambda,\omega}(z)=\sum_{n=0}^\infty X_n(\omega)\lambda_n z^n
\]
is analytic  on the unit disk.

For our proof, we need the following elementary inequalities which may be well known. However, we cannot locate literature.
For the sake of completeness, we have included a brief proof.
\begin{lemma}\label{ele}
	Let $(X_n)_{n\in\mathbb{N}}$ be a sequence of i.i.d. real random variables with $\mathbb{E}[X_n]=0$ and $\mathbb{E}[X_n^4]=1$. Then for all finite complex numbers $a_1,\cdots, a_n$, we have
	\[
	\mathbb{E}\bigg[\bigg|\sum_{i=1}^n a_i X_i\bigg|^4\bigg]\leq 3\bigg[\sum_{i=1}^n |a_i|^2\bigg]^2.
	\]
\end{lemma}

\begin{proof}
	For any fixed natural number $n\geq 1$,
	let
	$
	A_n=\sum_{i=1}^n a_i X_i.
	$
	It is easy to see
	\[
	|A_n|^2=\sum_{i=1}^n |a_i|^2X_i^2+2\sum_{i<j}Re(a_i\bar{a}_j)X_iX_j.
	\]
Note that $(X_n)_{n\in\mathbb{N}}$ is a sequence of i.i.d. real random variables with $\mathbb{E}[X_n]=0$  for each $n$. Then
	\begin{align*}
		\mathbb{E}\big[|A_n|^4\big]&=\mathbb{E}\left[\bigg(\sum_{i=1}^n |a_i|^2X_i^2\bigg)^2\right]+4\mathbb{E}\left[\bigg(
		\sum_{i<j}Re(a_i\bar{a}_j)X_iX_j
		\bigg)^2\right]\\
		 &=\sum_{i,j}|a_i|^2|a_j|^2\mathbb{E}[X_i^2X_j^2]+4\sum_{i<j}(Re(a_i\bar{a}_j))^2\mathbb{E}[X_i^2X_j^2]\\
		&\leq \sum_{i,j}|a_i|^2|a_j|^2\mathbb{E}[X_i^2X_j^2]+4\sum_{i<j}|a_i\bar{a}_j|^2\mathbb{E}[X_i^2X_j^2].
	\end{align*}
Observe that $\big(\mathbb{E}[X_i^2X_j^2]\big)^2\leq \mathbb{E}[X_i^4]\mathbb{E}[X_j^4]= 1$, hence
\[
\mathbb{E}\bigg[\bigg|\sum_{i=1}^n a_i X_i\bigg|^4\bigg]\leq 3\bigg[\sum_{i=1}^n|a_i|^2\bigg]^2,
\]
and the  proof is completed.
\end{proof}

Let $f\in H(\D)$. For $0<p<\infty$ and $0\leq r<1$, let
$$
M_p(f, r)=\left(\frac{1}{2\pi}\int_0^{2\pi}|f(re^{i\theta})|^p d\theta\right)^{\frac 1 p},
$$
and
$$
M_\infty (f, r)=\max_{\theta \in [0, 2\pi]}|f(re^{i\theta})|.
$$
Denote by $H^p$ the Hardy space consisting of those functions $f$ in $H(\D)$ with
$$
\|f\|_{H^p}=\sup_{0<r<1} M_p(f, r)<\infty.
$$
It is well known (cf. \cite[p. 3]{EKMR}) that $\d$ is a subset of $H^p$ for any $p>0$.

From  a result of L. Brown and A. Shields \cite[p. 300]{BS}, if an analytic function $f$ satisfies $M_p(f', r)\in L^2([0, 1], dr)$ with $2<p\leq \infty$, then
 $|f'(z)|^2dA(z)$ is a Carleson measure for $\d$. We  strengthen this conclusion as follows.

\begin{lemma}\label{bs}
Let   $2<p\leq \infty$. Suppose   $\phi$ is an analytic function  in the open unit disk $\D$ with $M_p(\phi', r)\in L^2([0, 1], dr)$. Then
$|\phi'(z)|^2dA(z)$ is a vanishing Carleson measure for the Dirichlet space.
\end{lemma}
\begin{proof}
  Let $2<p<\infty$.    Suppose  $\{f_m\}_{m=1}^\infty$ is  a sequence in $\d$ such that $\sup_m \|f_m\|_\d<\infty$ and functions $f_m$ tend to zero uniformly in compact subsets of $\D$ as $m\to \infty$.
Due to  $M_p(\phi', r)\in L^2([0, 1], dr)$, for any $\epsilon>0$, there exists a $\delta$ in $(0, 1)$ such that
$$
\int_\delta^1 \left(\int_0^{2\pi} |\phi'(re^{i\theta})|^p d\theta\right)^{\frac 2 p}dr<\epsilon.
$$
Then
\begin{align*}
 & \int_{\{z\in\D: |z|>\delta\}} |f_m(z)|^2 |\phi'(z)|^2 dA(z)\\
 \approx& \int_\delta^1   \int_0^{2\pi} |f_m(re^{i\theta})|^2 |\phi'(re^{i\theta})|^2  d\theta dr\\
 \lesssim& \left(\int_\delta^1  \left(\int_0^{2\pi} |\phi'(re^{i\theta})|^p d\theta\right)^{\frac 2 p}    dr \right)  \left(\int_\delta^1 \left(\int_0^{2\pi}  |f_m(re^{i\theta})|^{\frac{2p}{p-2}}  d\theta\right)^{\frac {p-2}{p}}  dr\right) \\
 \lesssim& \epsilon\ \|f_m\|_{H^{\frac{2p}{p-2}}}^2 \lesssim  \epsilon\ \|f_m\|_{\d}^2\lesssim  \epsilon.
 \end{align*}
 Since functions $f_m$ tend to zero uniformly in $\{z\in \D:  |z|\leq \delta\}$, there is a positive integer $N$ such that
	$$
	\int_{\{z\in \D:  |z|\leq \delta\}} |f_m(z) \phi'(z)|^2 dA(z) \leq \epsilon
	$$
	for all $m>N$. Consequently,
	$$
	\lim_{m\to \infty}\int_\D |f_m(z) \phi'(z)|^2 dA(z)=0.
	$$
	Thus  $|\phi'(z)|^2dA(z)$  is a vanishing Carleson measure for $\d$. The proof of the case of $p=\infty$ is similar. We omit it.

\end{proof}

Note that if $p=2$ and $\phi\in H(\D)$, then $M_p(\phi', r)\in L^2([0, 1], dr)$ if and only if $\phi \in \d$. In general, the measure
$|\phi'(z)|^2dA(z)$ with $\phi\in\mathcal{D}$ is not a vanishing Carleson measure for the Dirichlet space.

Now, we are ready to prove Theorem  \ref{random}.

\begin{proof}
If  $\mathbb{E}[X_n^4]=0$, then $X_n=0$ almost surely. Hence
without loss of generality, we can assume that $\mathbb{E}[X_n^4]=1$.
Recall that
$$
h_{\overline{\bm\lambda},\omega}(z)=\sum_{n=0}^\infty X_n(\omega)\overline{\lambda_n} z^n.
$$
By Theorem \ref{main00}, we shall show that almost surely
\[
|h_{\overline{\bm\lambda},\omega}'(z)|^2dA(z)
\]
is vanishing Carleson measure for $\d$.
By Lemma \ref{bs}, it is sufficient to show that
\[
\mathbb{E}\bigg[\int_0^1
\bigg(\int_0^{2\pi}|h_{\overline{\bm\lambda},\omega}'(re^{i\theta})|^4\frac{d\theta}{2\pi}\bigg)^{1/2}
dr\bigg]<\infty.
\]
It follows from  Fubini's theorem that
\begin{align*}
&\mathbb{E}\bigg[\int_0^1
\bigg(\int_0^{2\pi}|h_{\overline{\bm\lambda},\omega}'(re^{i\theta})|^4\frac{d\theta}{2\pi}\bigg)^{1/2}
dr\bigg]\\
=&\int_0^1\int_\Omega \bigg(\int_0^{2\pi}|h_{\overline{\bm\lambda},\omega}'(re^{i\theta})|^4\frac{d\theta}{2\pi}\bigg)^{1/2}d\mathbb{P}(w)dr.
\end{align*}
Applying the H\"older inequality with respect to the probability measure $d\mathbb{P}$, we have
\begin{align*}
&\int_0^1\int_\Omega \bigg(\int_0^{2\pi}|h_{\overline{\bm\lambda},\omega}'(re^{i\theta})|^4\frac{d\theta}{2\pi}\bigg)^{1/2}d\mathbb{P}(w)dr\\
\leq &\int_0^1\bigg( \int_\Omega \int_0^{2\pi}|h_{\overline{\bm\lambda},\omega}'(re^{i\theta})|^4\frac{d\theta}{2\pi}d\mathbb{P}(w)\bigg)^{1/2}dr.
\end{align*}
Consequently,
\[
\int_\Omega \int_0^{2\pi}|h_{\overline{\bm\lambda},\omega}'(re^{i\theta})|^4\frac{d\theta}{2\pi}d\mathbb{P}(w)
= \int_0^{2\pi}\int_\Omega|h_{\overline{\bm\lambda},\omega}'(re^{i\theta})|^4d\mathbb{P}(w)\frac{d\theta}{2\pi}.
\]
Recall that
\[
h_{\overline{\bm\lambda},\omega}'(re^{i\theta})=\sum_{n=1}^\infty \overline{\lambda_n} n r^{n-1}e^{i(n-1)\theta}X_n(\omega)
\]
and Lemma \ref{ele},  we get
\begin{align*}
&\int_0^1\left(\int_\Omega\bigg|\sum_{n=1}^\infty \overline{\lambda_n} n r^{n-1}e^{i(n-1)\theta}X_n\bigg|^4d\mathbb{P}(w)\right)^\frac12 dr\\
\leq& 3\sum_{n=1}^\infty  \int_0^1 | \overline{\lambda_n} n r^{n-1}|^2dr.
\end{align*}
Therefore, there is a constant $C$ such that
\begin{align*}
&\mathbb{E}\bigg[\int_0^1\bigg(\int_0^{2\pi}|h_{\bm \lambda,\omega}'(re^{i\theta})|^4\frac{d\theta}{2\pi}\bigg)^{1/2}dr\bigg]\\
\leq &
C\sum_{n=1}^\infty n|\lambda_n|^2<\infty.
\end{align*}
This completes the whole proof.
\end{proof}

To prove Corollary \ref{rudin}, we start with a simple observation.

\begin{lemma}\label{vcar}
	Suppose $p>2$, $E$ is a Rudin's $\Lambda(p)$ set and
	$\boldsymbol{\lambda}=\{\lambda_{n}: n\in E\}$  is a sequence of complex numbers. If $\sum_{n\in E}n|\lambda_n|^2<\infty$, then $|h_{\overline{\bm \lambda}}'(z)|^2dA(z)$ is a vanishing Carleson measure for the Dirichlet space, where
	$
	h_{\overline{\bm \lambda}}(z)=\sum_{n\in E}\overline{\lambda_n} z^n.
	$
\end{lemma}
\begin{proof}
	Let $p>2$. For any $0<r<1$,
	observe that
	\[
	h_{\overline{\bm \lambda}}'(re^{i\theta})=e^{-i\theta}\sum_{n\in E}\overline{\lambda_n} n r^{n-1}e^{in \theta},
	\]
	we have
	\[
	\|h_{\overline{\bm \lambda}}'(re^{i\theta})\|_{L^p([0, 2\pi], d\theta)}=
	\bigg\|\sum_{n\in E}\overline{\lambda_n} n r^{n-1}e^{in \theta}\bigg\|_{L^p([0, 2\pi], d\theta)}.
	\]
	Since $E$ is a Rudin's $\Lambda(p)$,
	\begin{align*}
	&\bigg\|\sum_{n\in E}\overline{\lambda_n} n r^{n-1}e^{in \theta}\bigg\|^2_{L^p([0, 2\pi], d\theta)}\\
	\leq & C \bigg\|\sum_{n\in E}\overline{\lambda_n} n r^{n-1}e^{in \theta}\bigg\|^2_{L^2([0, 2\pi], d\theta)}\\
	=&C \sum_{n\in E}|\lambda_n|^2n^2r^{2(n-1)}.
    \end{align*}
	Recall that $\sum_{n\in E}n|\lambda_n|^2<\infty$, therefore,
	\begin{align*}
		&\int_0^1 \bigg\|\sum_{n\in E}\overline{\lambda_n} n r^{n-1}e^{in \theta}\bigg\|^2_{L^p([0, 2\pi], d\theta)}dr\\
\leq &C \int_0^1 \sum_{n\in E}|\lambda_n|^2n^2r^{2(n-1)}dr\\
<& +\infty.
	\end{align*}
	The lemma then follows from Lemma \ref{bs}.
\end{proof}

\begin{proof}[Proof  of Corollary \ref{rudin}]
	Suppose  $\H_{\boldsymbol{\lambda}}$ is  bounded  on $\d$.  By Theorem \ref{main00}, $h_{\overline{{\bm \lambda}}}\in \mathbb{D}$. By the definition of Dirichlet norm, the condition  (iii) holds. $(i)\Rightarrow (ii)$ is clear. The implication $(iii)\Rightarrow (i)$ follows from Lemma \ref{vcar} and Theorem \ref{Main00c}. The proof is complete.
\end{proof}

\section{Final remarks}

In this section, we give some remarks about some functions in $\X$ and the action of Hankel matrices on the Bergman space $A^2$.

 As mentioned in \cite[p. 16]{ARSW0},
very little is known about functions in $\X$. Our results in this paper give     a satisfactory  description  of functions  in $\X$ with decreasing Taylor's  sequences  of positive  numbers. More precisely,
suppose ${\bm \lambda}=\{\lambda_n\}^\infty_{n=0}$  is a decreasing sequence of positive  numbers.  From Theorem \ref{main00} and Theorem \ref{H1}, the measure $|h'_{\bm \lambda}(z)|^2 dA(z)$ is a   Carleson measure for $\d$ (i.e. $h_{\bm \lambda}\in \X$) if and only if
    $$
   \sum^\infty_{n=m}n \lambda_n^2=O\left(\frac{1}{\log (m+2)}\right).
   $$
 By Theorem \ref{Main00c} and Theorem \ref{H2}, the measure $|h'_{\bm \lambda}(z)|^2 dA(z)$ is a vanishing   Carleson measure for $\d$ (i.e. $h_{\bm \lambda}\in \X_0$) if and only if
    $$
   \sum^\infty_{n=m}n \lambda_n^2=o\left(\frac{1}{\log (m+2)}\right).
   $$

 Suppose $\boldsymbol{\lambda}=\{\lambda_n\}_{n\in\mathbb{N}}$ is a sequence of complex numbers.
		Let $(X_n)_{n\in\mathbb{N}}$ be a sequence of i.i.d. real random variables with $\mathbb{E}[X_n]=0$ and $\mathbb{E}[X_n^4]<\infty$.
	If $h_{\bm \lambda}\in\mathcal{D}$,   Theorem \ref{random} and Theorem \ref{Main00c} yield that
$$
\mathbb{P}(h_{{\overline{{\bm \lambda}}, \omega}}\in \X)=\mathbb{P}(h_{\overline{{\bm \lambda}}, \omega}\in \X_0)=1,
$$
 where   $h_{\overline{{\bm \lambda}}, \omega}(z)=\sum_{n=0}^\infty X_n\overline{\lambda_n} z^n$.

Next  we consider   Hankel type operators  $\H_{\boldsymbol{\lambda}}$  on the Bergman space $A^2$ which  is also a Hilbert space of analytic functions on $\D$  and it is equipped  with the inner product
$$
\langle f, g\rangle_{A^2}=\int_\D f(z) \overline{g(z)} dA(z).
$$
From \cite[p. 349]{D1}, $(A^2)^*\cong \d$ and $\d^*=A^2$ under the  pairing
\begin{equation}\label{pair11}
\langle f, g \rangle= \sum^\infty_{k=0}a_k b_k,
\end{equation}
where   $f(z)=\sum^\infty_{k=0}a_kz^k$  and $g(z)=\sum^\infty_{k=0}b_k z^k$. For a sequence  $\boldsymbol{\lambda}=\{\lambda_n\}_{n\in\mathbb{N}}$  of complex numbers, it  is easy to see
$$
\langle \H_{\bm \lambda} f, g \rangle=\langle f,  \H_{\bm \lambda} g \rangle.
$$
Consequently, $\H_{\bm \lambda}$ is bounded (resp. compact) on $\d$ if and only if it is bounded (resp. compact) on $A^2$.
If we replace the pairing  (\ref{pair11}) by the Cauchy pairing
$$
(f, g )= \sum^\infty_{k=0}a_k\overline{ b_k}.
$$
Then
$$
( \H_{\bm \lambda} f, g )=( f,  \H_{\overline{\bm \lambda}} g ).
$$
Hence we also  get that $\H_{\bm \lambda}$ is bounded (resp. compact) on $\d$ if and only if $\H_{\overline{\bm \lambda}}$ is bounded (resp. compact) on $A^2$.

\vspace{0.1truecm}
\vspace{0.1truecm}
\noindent{\text{\bf Acknowledgement.}}
The work was supported by  National Natural Science Foundation of China (No. 12231005 and No. 12271328),
Guangdong Basic and Applied Basic Research Foundation (No. 2022A1515012117), China Postdoctoral Science Foundation (No. 2023TQ0074) and National Natural Science Foundation of of Chongqing (No. CSTB2022BSXM-JCX0088 and 2022NSCQ-MSX0321).

\end{document}